\documentclass[smallcondensed]{svjour3}

\smartqed

\usepackage{amsmath}
\usepackage{amssymb}
\usepackage{bbm}
\usepackage{graphicx}

\makeatletter
\newcommand{\leqnomode}{\tagsleft@true\let\veqno\@@leqno}
\newcommand{\reqnomode}{\tagsleft@false\let\veqno\@@eqno}
\makeatother

\begin{document}

\title{On Risk-Averse Stochastic Semidefinite Programs with Continuous Recourse
\thanks{The authors gratefully acknowledge the support of the German Research Foundation (DFG) within the collaborative research center TRR 154 ``Mathematical Modeling, Simulation and Optimization Using the Example of Gas Networks''.}
}


\titlerunning{Risk-Averse Stochastic Semidefinite Programs}

\author{Matthias Claus \and R\"udiger Schultz \and Kai~Sp\"urkel \and Tobias Wollenberg}

\institute{M. Claus \and R. Schultz \and K. Sp\"urkel \and T. Wollenberg \at University Duisburg-Essen \\ Thea-Leymann-Straße 9 \\ D-45127 Essen \\ Tel.: +49 201 183 6887 \\ \email{matthias.claus@uni-due.de}}

\date{Received: date / Accepted: date}

\maketitle

\begin{abstract}
The vast majority of the literature on stochastic semidefinite programs (stochastic SDPs) with recourse is concerned with risk-neutral models. In this paper, we introduce mean-risk models for stochastic SDPs and study structural properties as convexity and (Lipschitz) continuity. Special emphasis is placed on stability with respect to changes of the underlying probability distribution. Perturbations of the true distribution may arise from incomplete information or working with (finite discrete) approximations for the sake of computational efficiency. We discuss extended formulations for stochastic SDPs under finite discrete distributions, which turn out to be deterministic (mixed-integer) SDPs that are (almost) block-structured for many popular risk measures.

\end{abstract}

\section{Introduction}

Stochastic semidefinite programs with recourse were first considered by Ariyawansa and Zhu in \cite{AriyawansaZhu2006}, where, for finite discrete distributions, the authors reformulate the risk-neutral stochastic SDP as a block-structured deterministic SDP and discuss an application to the stochastic version of the minimum-volume covering ellipsoid problem (cf. \cite{SunFreund2004}, \cite{VandenbergheBoyd1996}). In \cite{ZhuAriyawansa2011}, the same authors give a multitude of other applications, including problems in geometry, location aided routing, RC circuit design and structural optimization.

\smallskip

Some approaches to the algorithmic treatment of  risk neutral programs with linear recourse carry over to expectation based stochastic SDPs. Extending the results of Zhao (cf. \cite{Zhao2001}), Mehrotra and \"Ozevin derive a polynomial logarithmic barrier algorithm employing Bender's decomposition (cf. \cite{MehrotraOezevin2007}). Using the volumetric barrier of Vaidya (cf. \cite{Vaidya1996}), Ariyawansa and Zhu construct algorithms of similar complexity in \cite{AriyawansaZhu2011}. Furthermore, in \cite{JinAriyawansaZhu2012}, Jin, Ariyawansa and Zhu propose homogeneous self-dual algorithms with complexities comparable to the ones of the methods mentioned before. Motivated by an application in multi-antenna wireless networks, Gaujal and Mertikopoulos establish a stochastic approximation algorithm in \cite{GaujalMertikopoulos2016}.

\smallskip

Chance constrained SDP models have been introduced by Ariyawansa and Zhu in \cite[Chapter 3]{Zhu2006}, where an application to the stochastic minimum-volume covering ellipsoid problem is considered. A different approach towards risk-aversion is taken by Schultz and Wollenberg, who consider stochastic mixed-integer semidefinite programs arising from unit commitment problems in AC transmission systems. Based on Lagrangian relaxation of the nonanticipativity constraint, a decomposition algorithm for minimizing a weighted sum of the expectation and the probability of exceeding a certain threshold is proposed in \cite{SchultzWollenberg2017}.

\smallskip

The present work extends the models of \cite{SchultzWollenberg2017} and \cite{AriyawansaZhu2011} by considering more general risk measures. Instead of focussing on a certain application, we discuss structural properties as convexity and (Lipschitz) continuity of the resulting objective functions. Consequences for quantitative stability of the stochastic SDP models under perturbations of the underlying distribution are pointed out. Such perturbations may arise from incomplete information about the distribution or the choice to work with a simpler (possibly finite discrete) approximation for reasons of computational efficiency.

\smallskip

Furthermore, we establish sufficient conditions for differentiabiliy in the risk neutral setting.
Finally, for finite discrete distributions, we establish equivalent SDPs for various risk measures and give indications on how to exploit their special structure for numerical treatment.

\section{Two-Stage Stochastic SDPs with Continuous Recourse}

Let $\mathcal{S}^k_+$ denote the cone of symmetric positive semidefinite matrices in $\mathbb{R}^{k \times k}$. The componentwise Frobenius product of $A = (a_1, \ldots a_s)^\top \in (\mathcal{S}^{k}_+)^l$ and $x \in \mathcal{S}^k_+$ is defined as $A \bullet x := \big( \mathrm{tr}(a_1 x), \ldots, \mathrm{tr}(a_s x) \big)^\top \in \mathbb{R}^s$. Furthermore, the Frobenius norm on $\mathcal{S}^k_+$ is given by $\|x\| := \sqrt{x \bullet x}$.

\smallskip

We shall consider the parametric SDP
\begin{equation}
\leqnomode
\label{ParametricSDP}
\tag{\textbf{P}($z$)}
\min_{x,y} \lbrace c \bullet x + q \bullet y \; | \; T \bullet x + W \bullet y = z, \; x \in X, \; y \in \mathcal{S}^m_+ \rbrace,
\end{equation}
where $z \in \mathbb{R}^s$ enters as a parameter. The data is comprised of $c \in \mathcal{S}^n_+$, $q \in \mathcal{S}^m_+$, $T \in (\mathcal{S}^n_+)^s$, $W \in (\mathcal{S}^m_+)^s$ and a nonempty, closed, convex set $X \subseteq \mathcal{S}^n_+$. The set $X$ is usually given as a spectrahedron, i.e. the intersection of the solution sets of a finite number of affine matrix inequalities with the cone of positive semidefinite matrices.

\smallskip

Let $z = Z(\omega)$ be the realization of a random vector $Z: \Omega \to \mathbb{R}^s$ on some probability space $(\Omega, \mathcal{F}, \mathbb{P})$. A two-stage stochastic SDP arises from \eqref{ParametricSDP} if the decision $x$ has to be taken without knowledge of the particular realization $Z(\omega)$, while $y$ can be chosen after observing the previously unknown parameter. In this setting, the optimal decision $y$ is governed by the recourse problem
\begin{equation}
\label{RecourseProblem}
\min_y \lbrace q \bullet y \; | \; W \bullet y = Z(\omega) - T \bullet x,\; y \in \mathcal{S}^m_+ \rbrace.
\end{equation}
Let $\varphi: \mathbb{R}^s \to \overline{\mathbb{R}}$  denote the optimal value function of \eqref{RecourseProblem} with respect to the right-hand side of the system of matrix equations in its constraints, i.e.
$$
\varphi(t) := \min_y \lbrace q \bullet y \; | \; W \bullet y = t,\; y \in \mathcal{S}^m_+ \rbrace.
$$

Introducing the function $f: \mathcal{S}^n_+ \times \mathbb{R}^s \to \overline{\mathbb{R}}$, $f(x,z) := c \bullet x + \varphi(z - T \bullet x)$ we may rewrite (\textbf{P}($Z(\cdot)$) as
\begin{equation}
\label{OptimazationOverRandomVariables}
\min_x \lbrace f(x,Z(\cdot)) \; | \; x \in X \rbrace.
\end{equation}
Due to the assumed interplay between decision and observation, problem \eqref{OptimazationOverRandomVariables} is not well-defined without further modelling choices. For any $x$, $f(x,Z(\cdot))$ belongs to the space $L^0(\Omega, \mathcal{F}, \mathbb{P})$ of extended real-valued random variables on the underlying probability space. We thus may fix any functional $\mathcal{R}: \mathcal{X} \to \overline{\mathbb{R}}$ satisfying
$$
\lbrace f(x, Z(\cdot)) \; | \; x \in X \rbrace \subseteq \mathcal{X} \subseteq L^0(\Omega, \mathcal{F}, \mathbb{P})
$$
and consider the optimization problem
\begin{equation}
\label{OptimizationWithRiskMeasures}
\min_x \lbrace Q_\mathcal{R}(x) \; | \; x \in X \rbrace,
\end{equation}
where the mapping $Q_\mathcal{R}: \mathcal{S}^n_+ \to \overline{\mathbb{R}}$ is given by $Q_\mathcal{R}(x) = \mathcal{R}[f(x,Z(\cdot))]$.

\smallskip

We shall work with the following assumptions:
\begin{itemize}
\item[A1] (\textbf{Complete recourse}) $W \bullet \mathcal{S}^m_+ = \mathbb{R}^s$.
\ \\ \vspace{-5pt}
\item[A2] (\textbf{Strict dual feasibility}) There is some $u \in \mathbb{R}^s$ such that $q - W^\top u$ is positive definite.
\end{itemize}

Similar, yet more restrictive assumptions are also made in \cite{MehrotraOezevin2007}.

\begin{lemma}
\label{LemmaMDKompakt}
Assume A2, then A1 holds if and only if $M_D := \lbrace u \in \mathbb{R}^s \; | \; q - W^\top u \in \mathcal{S}^m_+ \rbrace$ is compact.
\end{lemma}

\begin{proof}
$M_D$ is closed due to the closedness of $S^m_+$. Suppose that $M_D$ is unbounded, i.e. that there exists a sequence $\lbrace u_k \rbrace_{k \in \mathbb{N}} \subseteq M_D$ with $\lim_{k \to \infty} \|u_k\| = \infty$. Define $v_k:=u_k /\|u_k\|$, then ${\|v_k\|} = 1$ holds for all $k \in \mathbb{N}$. Therefore, the sequence $\lbrace v_k \rbrace_{k \in \mathbb{N}}$ can be assumed to converge to some $v \neq 0$ without loss of generality. By $u_k \in M_D$ we have $q -W^\top u_k \in \mathcal{S}^m_+$ for all $k \in \mathbb{N}$. Thus,
$$
-W^\top v = \lim_{k \to \infty} - W^\top v_k = \lim_{k \to \infty} \frac{1}{\|u_k\|} \big( q - W^\top u_k \big) \in \mathcal{S}^m_+.
$$
Now select any $u_0 \in M_D$. Then $u_0 + \alpha v \in M_D$ holds for any $\alpha \geq 0$ and we have
$$
\lim_{\alpha \to \infty} v^\top (u_0 + \alpha v) = \lim_{\alpha \to \infty} v^\top u_0 + \alpha \| v \|^2 = \infty,
$$
verifying $\sup \lbrace v^\top u \; | \; q - W^\top u \in \mathcal{S}^m_+ \rbrace = \infty$. By duality, the set $\lbrace y \in \mathcal{S}^m_+ \; | \; W \bullet y = v \rbrace$ has to be empty, which contradicts A1.

\smallskip

Let $M_D$ be compact, then once again by duality for arbitrary $t \in \mathbb{R}^s$, there exists $u \in M_D$ with $\min \lbrace q \bullet y \; | \; W \bullet y = t, \; y \in \mathcal{S}^m_+ \rbrace = t^\top u$, which implies $t \in W \bullet \mathcal{S}^m_+$ and thus A1. \qed
\end{proof}

The lemma above shows that $\sup \lbrace t^\top u \; | \; q - W^\top u \in \mathcal{S}^m_+ \rbrace$ is attained for any $t \in \mathbb{R}^{s}$ whenever A1 and A2 hold true.

\begin{lemma}
\label{LemmaRecourseFunctional}
Assume A1 and A2, then $\varphi$ is finite, convex and Lipschitz continuous on $\mathbb{R}^s$.
\end{lemma}

\begin{proof}
Due to A1 and A2, strong duality holds true for the SDP defining $\varphi$. We thus have
$$
\varphi(t)= \max_u \lbrace t^\top u \; | \; u \in M_D \rbrace \; \; \forall t \in \mathbb{R}^s.
$$
As $M_D$ is nonempty and compact by Lemma \ref{LemmaMDKompakt}, $\varphi$ is finite on $\mathbb{R}^s$.

\smallskip

Furthermore, for arbitrary $\lambda \in [0,1]$ and $t_1, t_2 \in \mathbb{R}^s$, strong duality implies
\begin{align*}
\varphi(\lambda t_1+ (1-\lambda)t_2) &=\max_{u \in M_D} (\lambda t_1+ (1-\lambda)t_2)^Tu \\
& \leq \lambda \max_{u \in M_D} t_1^Tu + (1-\lambda) \max_{u \in M_D} t_2^Tu \\
&= \lambda \varphi(t_1)+ (1-\lambda) \varphi(t_2),
\end{align*}
which proves the asserted convexity of $\varphi$.

\smallskip

To establish Lipschitz continuity, let $t_1, t_2 \in \mathbb{R}^s$ be arbitrary and fixed. Then by strong duality and the compactness of $M_D$, there exists $u_1,u_2 \in M_D$ such that $\varphi(t_1) = t_1^\top u_1$ and $\varphi(t_2) = t_2^\top u_2$. By $t_1^\top u_1 \geq t_1^\top u_2$ and $t_2^\top u_2 \geq t_2^\top u_1$ we have
$$
- \|u_2\| \cdot \|t_1-t_2\| \leq t_1^ \top u_2 - t_2^\top u_2 \leq \varphi(t_1) - \varphi(t_2) \leq t_1^ \top u_1 - t_2^\top u_1 \leq \| u_1 \| \cdot \|t_1-t_2 \|
$$
and thus $|\varphi(t_1)-\varphi(t_2)| \leq \max\{\|u_1\|,\|u_2\| \} \|t_1-t_2\|$. Set $L_\varphi := \max\nolimits_{u\in M_D}\|u\| < \infty$, then
$$
|\varphi(t_1)-\varphi(t_2)| \leq L_\varphi \cdot \|t_1-t_2\|
$$
holds for all $t_1, t_2 \in \mathbb{R}^s$, which completes the proof. \qed
\end{proof}

\begin{remark}
Under assumptions A1 and A2, $\varphi$ is finite and convex, which implies directional differentiability by \cite[Theorem 25.4]{Rockafellar1970}. Furthermore, the subdifferential of $\varphi$ is convex, compact and admits the representation
$$
\partial \varphi(t) = \mathrm{Argmax} \lbrace u^\top t \; | \; u \in M_D \rbrace.
$$
By \cite[Theorem 25.1]{Rockafellar1970}, $\varphi$ is differentiable at $t$ if and only if $\partial \varphi(t)$ is a singleton. In that case, we have $\partial \varphi(t) = \lbrace \nabla \varphi(t) \rbrace$.
\end{remark}

\begin{remark}
In two-stage stochastic linear programming, the counterpart of $\varphi$ is the optimal value function of a linear program:
$$
\varphi_l: \mathbb{R}^s \to \overline{\mathbb{R}}, \; \; \varphi_l(t) := \min \lbrace q_l^\top y_l \; | \; W_l y_l = t, \; y_l \in \mathbb{R}^m_+ \rbrace
$$
with $q_l \in \mathbb{R}^m$ and $W_l \in \mathbb{R}^{s \times m}$. By linear programming theory, $\varphi_l$ is finite on $\mathbb{R}^s$ iff $W_l (\mathbb{R}^m_+) =\mathbb{R}^s$ and $M_{D_l} = \lbrace u \in \mathbb{R}^s \; | \; W_l^\top u \leq q\} \neq \emptyset$. In this situation, $\varphi_l$ admits the representation
$$
\varphi_l(t) = \max_{j=1,...,N} d_j^\top t,
$$
where $d_1,...,d_N$ denote the vertices of the polytope $M_{D_l}$. In particular, $\varphi_l$ is piecewise linear, convex and Lipschitz continuous.
\end{remark}

The following example shows that the assumptions A1 and $M_D \neq \emptyset$ are not sufficient to ensure that the optimal value in the problem defining $\varphi(t)$ is attained for all $t \in \mathbb{R}^s$.

\begin{example}
For $t \in \mathbb{R}$, consider the SDP
\begin{equation}
\label{Ex1Primal}
\min \left\{ \begin{bmatrix} 1 & 0 & 0 & 0 \end{bmatrix} \bullet y \; | \; \begin{bmatrix} 0 & \frac{1}{2} & \frac{1}{2} & 0 \end{bmatrix} \bullet y = t, \; y \in \mathcal{S}^2_+ \right\}.
\end{equation}
For any $t \in \mathbb{R}$ we have 
$$
\begin{bmatrix} |t| +1 & t \\ t & |t| +1 \end{bmatrix} \in \mathrm{int} \; \mathcal{S}^2_+ \; \; \text{and} \; \; \begin{bmatrix} 0 & \frac{1}{2} & \frac{1}{2} & 0 \end{bmatrix} \bullet \begin{bmatrix} |t| +1 & t \\ t & |t| +1 \end{bmatrix} = t.
$$
Consequently, A1 is fulfilled. Moreover, we have
\begin{equation}
\label{Ex1MD}
M_D = \left\{ u \in \mathbb{R} \; | \; \begin{bmatrix} 1 & 0 & 0 & 0 \end{bmatrix} - \begin{bmatrix} 0 & \frac{1}{2} & \frac{1}{2} & 0 \end{bmatrix} \cdot u \in \mathcal{S}^2_+ \right\} = \lbrace 0 \rbrace.
\end{equation}
As \eqref{Ex1Primal} is strictly feasible for any right-hand side $t \in \mathbb{R}^s$, strong duality holds and \eqref{Ex1MD} implies that the infimum of \eqref{Ex1Primal} is zero. Furthermore, for any $t \in \mathbb{R} \setminus \lbrace 0 \rbrace$ we have
$$
\begin{bmatrix} y_{11} & \frac{t}{2} \\ \frac{t}{2}  & y_{22} \end{bmatrix} \in \mathcal{S}^2_+ \; \; \Leftrightarrow \; \; y_{11} > 0, \; y_{22} > 0, \; y_{11}y_{22}-\left( \frac{t}{2} \right)^2 \geq 0,
$$
which yields the lower bound $y_{11} \geq t^2/(4y_{22}) > 0$ for any $y$ that is feasible for \eqref{Ex1Primal}. Consequently, the optimal value in \eqref{Ex1Primal} is not attained if $t \neq 0$. 
\end{example}

\section{Structure of Risk-Averse Stochastic SDPs}

Let us now return to problem \eqref{OptimizationWithRiskMeasures} and consider various choices of $\mathcal{R}$. To ensure finiteness, we shall work with moment conditions on the Borel probability measure $\mathbb{P} \circ Z^{-1}$  induced by the underlying random vector $Z(\cdot)$. Let $\mathcal{P}(\mathbb{R}^s)$ denote the space of all Borel probability measures on $\mathbb{R}^s$ and
$$
\mathcal{M}^p_s := \lbrace \mu \in \mathcal{P}(\mathbb{R}^s) \; | \; \int_{\mathbb{R}^s} \|t\|^p~\mu(dt) < \infty \rbrace
$$
be the subspace of measures having finite moments of order $p \geq 1$.

\begin{lemma}
\label{LemmaF}
Assume A1, A2 and $\mathbb{P} \circ Z^{-1} \in \mathcal{M}^1_s$. Then $f(x,Z(\cdot)) \in L^1(\Omega, \mathcal{F}, \mathbb{P})$ for all $x \in \mathcal{S}^n_+$ and the mapping $F: \mathcal{S}^m_+ \to L^1(\Omega, \mathcal{F}, \mathbb{P})$, $F(x) := f(x,Z(\cdot))$ is convex and Lipschitz continuous with constant $\|c\| + L_\varphi \cdot \|T\|$.
\end{lemma}

\begin{proof}
For any $x \in \mathcal{S}^n_+$ we have
\begin{align*}
\|F(x)\|_{L^1} &= \int_{\mathbb{R}^s} |c \bullet x + \varphi(z - T \bullet x)|~(\mathbb{P} \circ Z^{-1})(dz) \\
&\leq |c \bullet x| + |\varphi(0)| + \int_{\mathbb{R}^s} |\varphi(z - T \bullet x) - \varphi(0)|~(\mathbb{P} \circ Z^{-1})(dz) \\
&\leq |c \bullet x| + |\varphi(0)| + L_\varphi \|T \bullet x\| + L_\varphi \int_{\mathbb{R}^s} \|z\|~(\mathbb{P} \circ Z^{-1})(dz) < \infty
\end{align*}
by Lemma \ref{LemmaRecourseFunctional}.

\smallskip

For any $x_1, x_2 \in \mathcal{S}^n_+$, $\lambda \in [0,1]$ and $z \in \mathbb{R}^s$, the convexity of $\varphi$ yields
$$
f(\lambda x_1 + (1-\lambda)x_2, z) \leq \lambda f(x_1,z) + (1-\lambda)f(x_2,z)
$$
and thus in particular $F(\lambda x_1 + (1-\lambda)x_2) \leq \lambda F(x_1) + (1-\lambda)F(x_2)$ with respect to the $\mathbb{P}$-almost sure partial order, proving the asserted convexity of $F$.

\smallskip

Finally,
\begin{align*}
\|F(x_1)-F(x_2)\|_{L^1} &= \int_{\mathbb{R}^s} |c \bullet (x_1 - x_2) + \varphi(z - T \bullet x_1) - \varphi(z - T \bullet x_2)|~(\mathbb{P} \circ Z^{-1})(dz) \\
& \leq \|c\| \cdot \|x_1-x_2\| + L_\varphi \cdot \|T\| \cdot \|x_1-x_2\|
\end{align*}
holds for all $x_1, x_2 \in \mathcal{S}^n_+$. \qed
\end{proof}

\begin{definition}
A mapping $\mathcal{R}: \mathcal{X} \to \mathbb{R} \cup \lbrace \infty \rbrace$ defined on some linear subspace $\mathcal{X}$ of $L^0(\Omega, \mathcal{F}, \mathbb{P})$ containing the constants is called a \textbf{convex risk measure} if the following conditions are fulfilled:
\begin{enumerate}
\item \textbf{(Convexity}) For any $Z_1, Z_2 \in \mathcal{X}$ and $\lambda \in [0,1]$ we have
$$
\mathcal{R}[\lambda Z_1 + (1-\lambda)Z_2] \leq \lambda \mathcal{R}[Z_1] + (1 - \lambda) \mathcal{R}[Z_2].
$$
\item (\textbf{Monotonicity}) $\mathcal{R}[Z_1] \leq \mathcal{R}[Z_2]$ for all $Z_1, Z_2 \in \mathcal{X}$ satisfying $Z_1 \leq Z_2$ with respect to the $\mathbb{P}$-almost sure partial order. \ \\ \vspace{-5pt}
\item (\textbf{Translation equivariance}) $\mathcal{R}[Z_1 + z_2] = \mathcal{R}[Z_1] + z_2$ for all $Z_1 \in \mathcal{X}$ and $z_2 \in \mathbb{R}$.
\end{enumerate}
A convex risk measure $\mathcal{R}$ is \textbf{coherent} if the following holds true:
\begin{enumerate}
\item[4.] (\textbf{Positive homogeneity}) $\mathcal{R}[z_2 Z_1] = z_2 \cdot \mathcal{R}[Z_1]$ for all $Z_1 \in \mathcal{X}$ and $z_2 \in [0,\infty)$.
\end{enumerate} 
\end{definition}

\begin{definition}\label{LawInvariance}
A mapping $\mathcal{R}: L^0(\Omega, \mathcal{F}, \mathbb{P}) \supseteq \mathcal{X} \to \mathbb{R} \cup \lbrace \infty \rbrace$ is called \textbf{law-invariant} if for all $Z_1, Z_2 \in L^0(\Omega, \mathcal{F}, \mathbb{P})$ with $\mathbb{P} \circ Z_1^{-1} = \mathbb{P} \circ Z_2^{-1}$ we have $\mathcal{R}[Z_1] = \mathcal{R}[Z_2]$.
\end{definition}

We shall give some examples of risk-measures frequently used in stochastic programming
as listed in \cite{RuszczynskiShapiro2003}, pp. 447-448, and \cite{ShapiroDentchevaRuszczynski2009}.
Later we will give extensive formulations of discrete mean-risk SDPs based on these risk-measures:  

\begin{itemize}
\item[(i)]
The expectation $\mathbb{E}: L^1(\Omega, \mathcal{F}, \mathbb{P}) \to \mathbb{R}$ 
is a law-invariant coherent risk-measure.\\

\item[(ii)]
The expected excess over threshold $\eta \in \mathbb{R}$ (as used in \cite{SchultzTiedemann2006})
is the mapping $\mathbb{EE}_\eta : L^1(\Omega, \mathcal{F}, \mathbb{P}) \rightarrow \mathbb{R}$
defined by
\begin{align*}
\mathbb{EE}_\eta[Y] = \int_\Omega \max\{ \, Y(\omega) - \eta, 0 \, \} \; \mathbb{P}(\text{d}\omega).
\end{align*}
This is a non-decreasing, convex and law-invariant risk measure, but in general not translation-equivariant.\\

\item[(iii)]
The conditional value-at-risk at level $\alpha \in \left( 0,1 \right)$
\begin{equation}
\label{DefCVAR}
\text{$\mathbb{CV}$@R}_\alpha: L^1(\Omega, \mathcal{F}, \mathbb{P}) \to \mathbb{R}, \; \text{$\mathbb{CV}$@R}_\alpha[Y] = \min_{\eta \in \mathbb{R}} \big\{ \, \eta + \frac{1}{1 - \alpha} \mathbb{EE}_\eta(Y) \, \big\}
\end{equation}
is law-invariant and coherent (cf. \cite{Pflug2000}).\\

\item[(iv)]
The value-at-risk at level $\alpha \in \left( 0,1 \right)$
$$
\text{$\mathbb{V}$@R}_\alpha: L^0(\Omega, \mathcal{F}, \mathbb{P}) \to \mathbb{R}, \;
\text{$\mathbb{V}$@R}_\alpha[Y] = \inf\{ \, t \ | \ \mathbb{P}( Z(\omega) \leq t ) 
\geq \alpha \, \}
$$
is nondecreasing, law-invariant, translation-equivariant and positively homogenous, but in general non-convex.\\

\item[(v)]
The upper semi-deviation of order $p$ is the mapping $\text{$\mathbb{M}$ad}^+_p : L^p(\Omega, \mathcal{F}, \mathbb{P}) \rightarrow \mathbb{R}$ defined by 
\begin{align*}
\text{$\mathbb{M}$ad}^+_p[Y] = \Big( \int \, \max\{ 0 , Y(\omega) - \mathbb{E}_\mathbb{P}[Z] \}^p 
\; \mathbb{P}(\text{d}\omega) \Big)^{\frac{1}{p}}.
\end{align*}
For $\rho \in \left[0,1\right]$ this gives rise to the law-invariant and coherent risk measure
$\mathbb{E} + \rho \, \mathbb{M}\text{ad}_p$
(cf. \cite{ShapiroDentchevaRuszczynski2009}, p. 276).\\

\end{itemize}

\begin{proposition} \label{PropConvexity}
Assume A1 and A2, let $\mathcal{X}$ be a convex subset of $L^0(\Omega, \mathcal{F}, \mathbb{P})$ that contains $F(\mathcal{S}^n_+)$ and fix a convex and nondecreasing mapping $\mathcal{R}: \mathcal{X} \to \mathbb{R}$. Then $Q_\mathcal{R}$ is finite and convex on $\mathcal{S}^n_+$. In particular, problem \eqref{OptimizationWithRiskMeasures} is convex.
\end{proposition}

\begin{proof}
Finiteness of $Q_\mathcal{R}$ follows directly from the finiteness of $\mathcal{R}$. Furthermore, for any $x_1, x_2 \in \mathcal{S}^n_+$ and $\lambda \in [0,1]$ we have
\begin{align*}
Q_\mathcal{R}(\lambda x_1 + (1-\lambda)x_2) &= \mathcal{R}[F(\lambda x_1 + (1-\lambda)x_2)] \\
&\leq \mathcal{R}[\lambda F(x_1) + (1-\lambda)F(x_2)] \\
&\leq \lambda \mathcal{R}[F(x_1)] + (1-\lambda)\mathcal{R}[F(x_2)].
\end{align*}
The first inequality above holds due to the monotonicity of $\mathcal{R}$ and the convexity of $F$ (by Lemma \ref{LemmaF}), while the second one is justified by the convexity of $\mathcal{R}$. \qed
\end{proof}

\begin{proposition}
Assume A1, A2 and that the support of $\mathbb{P} \circ Z^{-1}$ is bounded. Furthermore, let $\mathcal{R}: L^\infty(\Omega, \mathcal{F}, \mathbb{P}) \to \mathbb{R} \cup \lbrace \infty \rbrace$ be a coherent risk measure and assume that there is some $Y \in L^\infty(\Omega, \mathcal{F}, \mathbb{P})$ such that $\mathcal{R}[Y] < \infty$. Then $Q_\mathcal{R}$ is finite and Lipschitz continuous with constant $\|c\| + L_\varphi \cdot \|T\|$ on $\mathcal{S}^n_+$.
\end{proposition}

\begin{proof}
$\mathcal{R}$ is finite and Lipschitz continuous with constant $1$ with respect to the $L^\infty$-norm on by $L^\infty(\Omega, \mathcal{F}, \mathbb{P})$ by \cite[Lemma 4.3]{FoellmerSchied2004}.

\smallskip

For any $x \in \mathcal{S}^n_+$, the mapping $f(x,\cdot)$ is continuous by Lemma \ref{LemmaRecourseFunctional}, which implies
$$
\sup_{z \in \mathrm{supp}(\mathbb{P} \circ Z^{-1})}|f(x,z)| < \infty.
$$
Thus, $F(\mathcal{S}^n_+) \subseteq  L^\infty(\Omega, \mathcal{F}, \mathbb{P})$, which implies the asserted finiteness of $Q_\mathcal{R}$.

\smallskip

Furthermore, for any $x_1, x_2 \in \mathcal{S}^n_+$, we have
\begin{align*}
|Q_\mathcal{R}(x_1) - Q_\mathcal{R}(x_2)| &= |\mathcal{R}[F(x_1)] - \mathcal{R}[F(x_2)]| \\
&\leq \|F(x_1) - F(x_2)\|_{L^\infty} \\
&\leq \|F(x_1) - F(x_2)\|_{L^1} \\
&\leq (\|c\| + L_\varphi \cdot \|T\|) \cdot \|x_1 - x_2\|
\end{align*}
by Lemma \ref{LemmaF}. \qed
\end{proof}

If the support of $\mathbb{P} \circ Z^{-1}$ is unbounded, $F(\mathcal{S}^n_+)$ may fail to be a subset of $L^\infty(\Omega, \mathcal{F}, \mathbb{P})$. While Lipschitz continuity with respect to any $L^p$-norm with $p < \infty$ does not hold for general coherent risk measures, the Conditional Value-at-Risk $\mathrm{CVaR}_\alpha$ is known to be Lipschitz continuous with respect to the $L^1$-norm with constant $\frac{1}{1-\alpha}$ (cf. \cite[Corollary 3.7]{Pichler2017}). Using the Kusuoka representation (cf. \cite{Kusuoka2001}), this allows to replace the boundedness of the support of $\mathbb{P} \circ Z^{-1}$ with a less restrictive assumption on the moments of $\mathbb{P} \circ Z^{-1}$ for special classes of risk measures.

\begin{definition}
Random variables $Z_1$ and $Z_2$ are called \textbf{comonotonic} if $(Z_1,Z_2)$ is distributionally equivalent to $(F^{-1}_{Z_1}(U), F^{-1}_{Z_2}(U))$ where $U$ is uniformly distributed on $[0,1]$.

\smallskip

A coherent risk measure $\mathcal{R}: \mathcal{X} \to \mathbb{R}$ is said to be \textbf{comonotonic} if for any two comonotonic random variables $Z_1, Z_2 \in \mathcal{X}$ we have $\mathcal{R}(Z_1 + Z_2) = \mathcal{R}(Z_1) + \mathcal{R}(Z_2)$.
\end{definition}

For a discussion of comonotonicity we refer to \cite{DhaeneEtAl2002} and \cite{DhaeneEtAl2006}. A proof of the following result is given in \cite[Theorem 2]{Shapiro2013}:

\begin{theorem}
\label{TheoremComonotonicityShapiro}
A law-invariant coherent risk measure $\mathcal{R}: L^p(\Omega, \mathcal{F}, \mathbb{P}) \to \mathbb{R}$ with \\ $p \in [1,\infty)$ is comonotonic if and only if there exists probability measure $\nu$ on $[0,1)$ such that
\begin{equation}
\label{ComonotonicKusuoka}
\mathcal{R}(Y) = \int_{0}^{1} \text{$\mathbb{CV}$@R}_\alpha(Y)~\nu(d \alpha)
\end{equation}
holds for all $Y \in L^p(\Omega, \mathcal{F}, \mathbb{P})$. Furthermore, the measure $\nu$ in representation \eqref{ComonotonicKusuoka} is defined uniquely.
\end{theorem}

\begin{example} Using $\delta_{\alpha_0}$ to denote the Dirac measure at $\alpha_0 \in [0,1)$
$$
\text{$\mathbb{CV}$@R}_{\alpha_0}(Y) = \int_{0}^1 \text{$\mathbb{CV}$@R}_\alpha(Y)~\delta_{\alpha_0}(d \alpha)
$$
and, in particular,
$$
\mathbb{E}[Y] = \int_{0}^1 \text{$\mathbb{CV}$@R}_\alpha(Y)~\delta_{0}(d \alpha)
$$
hold for all $Y \in L^1(\Omega, \mathcal{F}, \mathbb{P})$.
\end{example}

\begin{proposition}
Let $\mathcal{R}: L^p(\Omega, \mathcal{F}, \mathbb{P}) \to \mathbb{R}$ with $p \in [1,\infty)$ be a law-invariant, comonotonic coherent risk measure. Assume A1, A2, $\mathbb{P} \circ Z^{-1} \in \mathcal{M}^p_s$ and
$$
L_\nu := \int_{0}^{1} \frac{1}{1-\alpha}~\nu(d \alpha) < \infty,
$$
where $\nu$ denotes the uniquely defined probability measure form representation \eqref{ComonotonicKusuoka}. Then $Q_\mathcal{R}$ is Lipschitz continuous with constant $L_\nu \cdot (\|c\| + L_\varphi \cdot \|T\|)$ on $\mathcal{S}^n_+$.
\end{proposition}

\begin{proof}
For any $x_1, x_2 \in \mathcal{S}^n_+$, we have
\begin{align*}
|Q_\mathcal{R}(x_1) - Q_\mathcal{R}(x_2)| &\leq \int_{0}^{1} |\mathrm{CVaR}_\alpha(F((x_1)) - \mathrm{CVaR}_\alpha(F((x_2))|~\nu(d \alpha) \\
&\leq \int_{0}^{1} \frac{1}{1-\alpha} \cdot \|F((x_1)) - F((x_2))\|_{L^1}~\nu(d \alpha) \\
&\leq \int_{0}^{1} \frac{1}{1-\alpha} \cdot (\|c\| + L_\varphi \cdot \|T\|) \cdot \|x_1 - x_2\|~\nu(d \alpha) \\
&= L_\nu \cdot (\|c\| + L_\varphi \cdot \|T\|) \cdot \|x_1 - x_2\|.
\end{align*}
The second inequality above holds due to \cite[Corollary 3.7]{Pichler2017}, while the third one is justified by Lemma \ref{LemmaF}. \qed
\end{proof}

We shall now study the dependence of $Q_\mathcal{R}$ on the underlying probability measure $\mathbb{P} \circ Z^1$. This is motivated by the fact that in applications the true probability distribution of the random parameter may be unknown. In such situations, one may work with an approximation if the optimal value function and the optimal solution set mapping of \eqref{OptimizationWithRiskMeasures} are at least semicontinuous with respect to changes of the underlying distribution.

\smallskip

Let $(\Omega_0, \mathcal{F}_0, \mathbb{P}_0)$ be an atomless probability space, i.e. assume that for any $A \in \mathcal{F}_0$ with $\mathbb{P}_0(A) > 0$ there exists some $B \subsetneq A$ with $B \in \mathcal{F}_0$ and $\mathbb{P}_0(B) > 0$, and fix any $p \geq 1$. Then for any $\nu \in \mathcal{M}^1_p$  there exists some $Z_\nu \in L^p(\Omega_0, \mathcal{F}_0, \mathbb{P}_0)$ such that $\mathbb{P}_0 \circ Z_\nu^{-1}$. Thus, given any law-invariant mapping $\mathcal{R}_0: L^p(\Omega_0, \mathcal{F}_0, \mathbb{P}_0) \to \mathbb{R}$, the function
$$
\Theta_{\mathcal{R}_0}: \mathcal{M}^1_p \to \mathbb{R}, \; \; \Theta_{\mathcal{R}_0}[\nu] := \mathcal{R}_0[Z_\nu]
$$
is well-defined. Furthermore, we can construct a mapping $\mathcal{R}_{\mathcal{R}_0}:  L^p(\Omega, \mathcal{F}, \mathbb{P}) \to \mathbb{R}$ by setting $\mathcal{R}_{\mathcal{R}_0}[Z_1] := \Theta_{\mathcal{R}_0}[\mathbb{P} \circ Z_1^{-1}]$. To ease the notation, we shall assume that $(\Omega, \mathcal{F}, \mathbb{P})$ itself is atomless. Given any law-invariant mapping $\mathcal{R}:  L^p(\Omega, \mathcal{F}, \mathbb{P}) \to \mathbb{R}$, we shall consider the function
$$
\mathcal{Q}_\mathcal{R}: \mathcal{S}^n_+ \times \mathcal{M}^p_s \to \mathbb{R}, \; \; \mathcal{Q}_\mathcal{R}(x,\mu) := \Theta_\mathcal{R}[\mu \circ f(x,\cdot)^{-1}].
$$
For the following analysis, we equip the space $\mathcal{P}(\mathbb{R}^s)$ with the topology of weak convergence, where a sequence $\lbrace \mu_k \rbrace_{k \in \mathbb{N}} \subseteq \mathcal{P}(\mathbb{R}^s)$ converges to some $\mu \in \mathcal{P}(\mathbb{R}^s)$, written $\mu_k  \stackrel{w}{\rightarrow} \mu$ if and only if
$$
\int_{\mathbb{R}^s} h(t)~\mu_k(dt) \rightarrow \int_{\mathbb{R}^s} h(t)~\mu(dt)
$$
holds for any bounded and continuous function $h: \mathbb{R}^s \to \mathbb{R}$. It is well known that even for linear recourse one cannot expect weak continuity of $\mathcal{Q}_\mathcal{R}$ on the entire space $\mathcal{S}^n_+ \times \mathcal{M}^p_s$. Along the lines of \cite{ClausKraetschmerSchultz2017}, we shall thus restrict the analysis to appropriate subspaces.

\begin{definition}
A set $\mathcal{M} \subseteq \mathcal{M}^p_s$ is called \textbf{locally uniformly $\|\cdot\|^p$-integrating} if for any $\mu \in \mathcal{M}$ and any $\epsilon > 0$ there exists some open neighborhood $\mathcal{N}$ of $\mu$ with respect to the topology of weak convergence such that
$$
\lim_{a \to \infty} \; \sup_{\nu \in \mathcal{N} \cap \mathcal{M}} \int_{\mathbb{R}^s} \mathbbm{1}_{(a,\infty)}(\|t\|^p) \cdot \|t\|^p~\nu(dt) \leq \epsilon. 
$$
\end{definition}

\begin{example}
\label{ExampleLocUnifIntSets}
(a) For any $K, \epsilon > 0$ and $p \geq 1$, the set
$$
U(\epsilon, K) := \lbrace \nu \in \mathcal{M}^p_s : \int_{\mathbb{R}^s} \|t\|^{1+\epsilon}~\nu(dt) \leq K \rbrace
$$
of measures having uniformly bounded moments of order $1 + \epsilon$ is locally uniformly $\|\cdot|^p$-integrating (cf. \cite[Lemma 2.69]{Claus2016}).

\smallskip

(b) For any $p \geq 1$ and compact set $\Xi \subset \mathbb{R}^s$, the set 
$$
\lbrace \nu \in \mathcal{M}^p_s : \int_{\Xi} 1~\nu(dt) = 1 \rbrace
$$
of measures with support in $\Xi$ is locally uniformly $\|\cdot\|^p$-integrating by \cite[Lemma 5.1]{KraetschmerSchiedZaehle2017}.

\smallskip

(c) Any singleton $\lbrace \mu \rbrace \subseteq \mathcal{M}^p_s$ is locally uniformly $\|\cdot\|^p$-integrating for any $p \geq 1$ by \cite[Lemma 5.2]{KraetschmerSchiedZaehle2017}.
\end{example}

\begin{theorem}
\label{TheoremStability}
Let $\mathcal{R}: L^p(\Omega, \mathcal{F}, \mathbb{P}) \to \mathbb{R}$ with $p \geq 1$ be law-invariant, convex and nondecreasing. Assume A1 and A2 and let $\mathcal{M} \subseteq \mathcal{M}^p_s$ be locally uniformly $\| \cdot \|^p$-integrating. Then the following statements hold true:
\begin{enumerate}
\item The restriction of $\mathcal{Q}_\mathcal{R}$ to the set $\mathcal{S}^n_+ \times \mathcal{M}$ is continuous with respect to the product topology of the the standard topology on $\mathcal{S}^n_+$ and the relative topology of weak convergence on $\mathcal{M}$.
\item The optimal value function
$$
\phi: \mathcal{M} \to \overline{\mathbb{R}}, \; \; \phi(\mu) := \min_x \lbrace \mathcal{Q}_\mathcal{R}(x,\mu) \; | \; x \in X \rbrace 
$$
is weakly upper semicontinuous.
\end{enumerate}
Additionally assume that $X$ is compact. Then
\begin{enumerate}
\item[3.] $\phi$ is weakly continuous. 
\item[4.] The optimal solution set mapping
$$
\Phi: \mathcal{M} \rightrightarrows \mathcal{S}^n_+, \; \; \Phi(\mu) := \mathrm{Argmin}_x \lbrace \mathcal{Q}_\mathcal{R}(x,\mu) \; | \; x \in X \rbrace 
$$
is weakly upper semicontinuous in the sense of Berge, i.e. for any $\mu_0 \in \mathcal{M}$ and any open set $\mathcal{O} \subseteq \mathcal{S}^{n}_+$ with $\Phi(\mu_0) \subseteq \mathcal{O}$ there exists a weakly open neighborhood $\mathcal{N}$ of $\mu_0$ such that $\Phi(\mu) \subseteq \mathcal{O}$ for all $\mu \in \mathcal{N} \cap \mathcal{M}$.
Furthermore, $\Phi(\mu)$ is nonempty and compact for any $\mu \in \mathcal{M}$.
\end{enumerate}
\end{theorem}

\begin{proof}
Invoking Lemma \ref{LemmaRecourseFunctional}, the result follows from \cite[Corollary 2]{ClausKraetschmerSchultz2017}. \qed
\end{proof}

\begin{corollary}
Let $\mathcal{R}: L^p(\Omega, \mathcal{F}, \mathbb{P}) \to \mathbb{R}$ with $p \geq 1$ be law-invariant, convex and nondecreasing and assume A1 and A2. Then $Q_\mathcal{R}$ is continuous.
\end{corollary}

\begin{proof}
By part (c) of Example \ref{ExampleLocUnifIntSets} we may apply the first part of Theorem \ref{TheoremStability} to $\mathcal{M} = \lbrace \mathbb{P} \circ Z^{-1} \rbrace$. The asserted continuity follows from $Q_\mathcal{R}(x) = \mathcal{Q}_\mathcal{R}(x, \mathbb{P} \circ Z^{-1})$ for any $x \in \mathcal{S}^n_+$. \qed
\end{proof}

We shall now turn our attention to questions of differentiability, but confine the analysis to the risk neutral model.

\begin{lemma}
Assume A1, A2 and $\mathbb{P} \circ Z^{-1} \in \mathcal{M}^1_s$, then the functional $Q_\mathbb{E}: \mathcal{S}^n_+ \to \mathbb{R}$, $Q_\mathbb{E}(x) := \mathbb{E}[F(x)]$ is directionally differentiable and
$$
Q_\mathbb{E}'(x; v) := \int_{\mathbb{R}^s} \varphi'(z-T \bullet x;v)~(\mathbb{P} \circ Z^{-1})(dz)
$$
holds for all $x, v \in \mathcal{S}^n_+$.
\end{lemma}

\begin{proof}
$Q_\mathbb{E}$ is finite valued by Lemma \ref{LemmaF}, convex by Proposition \ref{PropConvexity} and thus directionally differentiable (cf. \cite[Theorem 25.4]{Rockafellar1970}). Furthermore, $\varphi'(\cdot-Tx;v)$ is a pointwise limit of measurable functions and thus measurable for any $x, v \in \mathcal{S}^n_+$. The asserted representation of the directional derivative is justified by Lemma \ref{LemmaRecourseFunctional} and \cite[Proposition 2.1]{Bertsekas1973}. \qed
\end{proof}

Sufficient conditions for differentiability $Q_\mathbb{E}$ can be obtained using the same arguments as for linear recourse (cf. \cite{ShapiroDentchevaRuszczynski2009}).

\begin{lemma}
Assume A1, A2 and $\mathbb{P} \circ Z^{-1} \in \mathcal{M}^1_s$ and let $x_0 \in \mathcal{S}^n_+$ be such that
$$
\mathrm{Argmax} \lbrace u^\top (z - T \bullet x_0) \; | \;  u \in M_D \rbrace
$$
is a singleton for $(\mathbb{P} \circ Z^{-1})$-almost all $z \in \mathbb{R}^s$. Then $Q_\mathbb{E}$ is differentiable at $x_0$.
\end{lemma}

\begin{proof}
For $(\mathbb{P} \circ Z^{-1})$-almost all $z \in \mathbb{R}^s$, $h_z: \mathcal{S}^n_+ \to \mathbb{R}$, $h_z(x) = c \bullet x + \varphi(z - T \bullet x)$ is differentiable with measurable derivative 
$$
h_z'(x) = c + - T^\top \cdot \mathrm{Argmax} \lbrace u^\top (z - T \bullet x_0) \; | \;  u \in M_D \rbrace.
$$
Consider the functions $g_z: \mathcal{S}^n_+ \to \mathbb{R}$ defined by
$$
g_z(x) := \frac{h_z(x) - h_z(x_0) - h_z'(x_0)^\top(x-x_0)}{\|x-x_0\|},
$$
then $\lim_{x \to x_0} g_z(x) = 0$ holds for $(\mathbb{P} \circ Z^{-1})$-almost all $z \in \mathbb{R}^s$. Furthermore, Lemma \ref{LemmaRecourseFunctional} implies $\|g_z(x)\| \leq 2(L_\varphi\|T\| + \|c\|)$ for all $x \in \mathcal{S}^n_+$ and $z \in \mathbb{R}^s$. Hence, by Lebesgue's dominated convergence theorem, we have
\begin{align*}
&\lim_{x \to x_0} \frac{Q_\mathbb{E}(x) - Q_\mathbb{E}(x_0) - \int_{\mathbb{R}^s} h'_z(x_0)^\top(x-x_0)~(\mathbb{P} \circ Z^{-1})(dz)}{\|x-x_0\|} \\
= &\lim_{x \to x_0} \int_{\mathbb{R}^s} g_z(x)~(\mathbb{P} \circ Z^{-1})(dz) = \int_{\mathbb{R}^s} \lim_{x \to x_0} g_z(x)~(\mathbb{P} \circ Z^{-1})(dz) = 0.
\end{align*}
Consequently, $Q_\mathbb{E}$ is differentiable at $x_0$ and $Q_\mathbb{E}'(x_0) = \int_{\mathbb{R}^s} h_z'(x_0)~(\mathbb{P} \circ Z^{-1})(dz)$. \qed
\end{proof}

\begin{corollary}
Assume A1, A2 and that $\mathbb{P} \circ Z^{-1} \in \mathcal{M}^1_s$ is absolutely continuous with respect to the Lebesgue measure. Then $Q_\mathbb{E}$ is continuously differentiable on $\mathcal{S}^n_+$.
\end{corollary}

\begin{proof}
Let $N_\varphi \subset \mathbb{R}^s$ denote the set of points of nondifferentiability of $\varphi$. By \cite[Theorem 25.5]{Rockafellar1970},
$$
N_x := \lbrace z \in \mathbb{R}^s \; | \; z - T \bullet x \in N_\varphi \rbrace
$$
is a null set with respect to the Lebesgue measure for any $x \in \mathcal{S}^n_+$, which implies $(\mathbb{P} \circ Z^{-1})[N_x] = 0$. Consequently, $Q_\mathbb{E}$ is differentiable on $\mathcal{S}^n_+$. Continuity of the derivative follows from \cite[Theorem 25.5]{Rockafellar1970} and the convexity of $Q_\mathbb{E}$. \qed
\end{proof}

\begin{remark}
Assuming A1, A2 and $\mathbb{P} \circ Z^{-1} \in \mathcal{M}^1_s$, the subdifferential of $Q_\mathbb{E}$ admits the representation
\begin{align*}
&\partial Q_\mathbb{E}(x) = c + \int_\mathbb{R}^s \partial_x \varphi(z-T \bullet x)~(\mathbb{P} \circ Z^{-1})(dz) \\
&= \Big\{ c + \int_{\mathbb{R}^s} \rho(z)~(\mathbb{P} \circ Z^{-1})(dz) \; | \; \rho: \mathbb{R}^s \to \mathcal{S}^n_+ \; \text{measurable}, \; \rho(z) \in \partial_x \varphi(z-T \bullet x) \; \text{a.s.} \Big\}.
\end{align*}
Furhter details are given in \cite{Bertsekas1973}.
\end{remark}

\begin{corollary}
Assume A2 and that the underlying random variable $Z$ follows a finite discrete distribution with realizations $z_1, \ldots, z_S \in \mathbb{R}^s$ and respective probabilities $\pi_1, \ldots, \pi_S > 0$. Furthermore, assume that $\lbrace y \in \mathcal{S}^m_+ \; | \; W \bullet y = z_i - T \bullet x \rbrace$ is nonempty for any $i \in \lbrace 1, \ldots, S \rbrace$ and $x \in \mathcal{S}^n_+$. Then
\begin{align*}
\partial Q_\mathbb{E}(x) &= c + \sum_{i=1}^{s} \pi_i \cdot \partial_x \varphi(z_i - T \bullet x) \\
&= c + \sum_{i=1}^{s} -\pi_i \cdot T^\top \cdot \mathrm{Argmax} \lbrace u^\top (z_i - T \bullet x) \; | \;  u \in M_D \rbrace
\end{align*}
holds for any $x \in \mathcal{S}^n_+$.
\end{corollary}

\begin{proof}
The result follows directly from \cite[Theorem 23.8]{Rockafellar1970}. \qed
\end{proof}

\section{Extensive Formulations for Finite Discrete Distributions}
Throughout this section, we shall assume A1, A2 and that the underlying random variable $Z$ follows a finite discrete distribution with realizations $z_1, \ldots, z_S \in \mathbb{R}^s$ and respective probabilities $\pi_1, \ldots, \pi_S > 0$. Furthermore, we denote the index set $\lbrace 1, \ldots, S \rbrace$ by $\mathcal{I}_S$.

\smallskip

It is well known that in the risk neutral setting, the stochastic SDP admits a reformulation as a block-structured SDP (cf. \cite{AriyawansaZhu2006}, \cite{MehrotraOezevin2007}):

\begin{proposition} \label{PropRiskNeutralReformulation}
The risk neutral stochastic SDP
\begin{equation}
\label{RiskNeutralSSDP}
\min \left\{ Q_\mathbb{E}(x) \; | \; x \in X \right\}
\end{equation}
is equivalent to the SDP
\begin{align}
\label{RiskNeutralSSDPReformulation}
\min_{x, y_1, \ldots, y_S} \bigg\{ c \bullet x + \sum_{i = 1}^{S} \pi_i q \bullet y_i \; | \; T \bullet x + W \bullet y_i = z_i \; \forall i \in \mathcal{I}_S&,
\\ x \in X, \; y_i \in \mathcal{S}^m_+ \; \forall i \in \mathcal{I}_S& \bigg\} \nonumber,
\end{align}
in the sense that the infimal values of the problems coincide. Furthermore, $x$ is an optimal solution for \eqref{RiskNeutralSSDP} if and only if there exist $v$ and $y_1, \ldots, y_S$ such that $(x,v,y_1, \ldots, y_S)$ is an optimal solution for \eqref{RiskNeutralSSDPReformulation}.
\end{proposition}

\begin{proof}
By definition of $\varphi$,
\begin{equation}
\label{ProofEquivExp}
Q_\mathbb{E}(x) = c \bullet x + \sum_{i = 1}^{S} \pi_i \varphi(z_i - T \bullet x) \leq c \bullet x + \sum_{i = 1}^{S} \pi_i q \bullet y_i
\end{equation}
holds for any $x \in X$, $y_1, \ldots, y_S \in \mathcal{S}^m_+$ satisfying $T \bullet x + W \bullet y_i = z_i$ for all $i \in \mathcal{I}_S$. Thus, the infimal value of \eqref{RiskNeutralSSDP} is less or equal to the infimal value of \eqref{RiskNeutralSSDPReformulation}. Furhtermore, \eqref{ProofEquivExp} is satisfied as equality if and only if
$$
y_i \in \mathrm{Argmin} \lbrace q \bullet y \; | \; T \bullet x + W \bullet y = z_i, \; y \in \mathcal{S}^m_+ \rbrace
$$
holds for all $i \in \mathcal{I}_S$. The optimal solution set above is nonempty by strong duality, which holds due to A1 and A2. \qed
\end{proof}

We continue with extensive formulations of the SDP \eqref{OptimizationWithRiskMeasures} for mean-risk models based on the risk measures immediately following Definition \ref{LawInvariance}. In this context, $\rho$ shall always be a nonnegative, predefined parameter indicating risk-aversion in the optimization.

\begin{proposition}\label{ExtensiveMeanRiskEE}
\begin{equation}
\label{ExpExcessSSDP}
\min \left\{Q_{\mathbb{E} + \rho \, \mathbb{EE}_{\eta}}(x) \; | \; x \in X \right\},
\end{equation}
with $\eta \in \mathbb{R}$ as a given parameter, can be equivalently restated as
\begin{align}
\label{ExpEscessSSDPReformulation}
\min_{\substack{x, v_1, \ldots, v_S, \\ y_1, \ldots, y_S}} \bigg\{ c \bullet x + \sum_{i = 1}^{S} \pi_i q \bullet y_i + \rho \, \sum_{i = 1}^{S} \pi_i v_i \; | \; T \bullet x + W \bullet y_i = z_i \; \forall i \in \mathcal{I}_S&, \\ \nonumber
v \geq 0, \; v_i \geq c \bullet x + q \bullet y_i - \eta \; \forall i \in \mathcal{I}_S&, \\
x \in X, \; y_i \in \mathcal{S}^m_+ \; \forall i \in \mathcal{I}_S& \bigg\}. \nonumber
\end{align}
\end{proposition}

\begin{proof}
As the objective function of \eqref{ExpEscessSSDPReformulation} is increasing with respect to $v$, any optimal solution $(x, v_1, \ldots, v_S, y_1, \ldots, y_S)$ satisfies $v_i = \max \lbrace c \bullet x + q \bullet y_i - \eta, 0 \rbrace$ for all $i \in \mathcal{I}_S$. The asserted equivalence of \eqref{ExpExcessSSDP} and \eqref{ExpEscessSSDPReformulation} then follows as in the proof of Proposition~\ref{PropRiskNeutralReformulation}. \qed
\end{proof}

\begin{proposition}
\begin{equation*}
\min \left\{Q_{\mathbb{E} + \rho \, \mathbb{CV}@R_{\alpha}}(x) \; | \; x \in X \right\}
\end{equation*}
can be equivalently restated as
\begin{align}
\label{CVaRReformulation}
\min_{\substack{x, v_1, \ldots, v_S, \\ y_1, \ldots, y_S, \eta}} \bigg\{ c \bullet x + \sum_{i = 1}^{S} \pi_i q \bullet y_i + \rho \, \eta + \frac{\rho}{1-\alpha} \,  \sum_{i = 1}^{S} \pi_i v_i  \; | \\
T \bullet x + W \bullet y_i = z_i \; \forall i \in \mathcal{I}_S&, \nonumber \\ 
v \geq 0, \; v_i \geq c \bullet x + q \bullet y_i - \eta \; \forall i \in \mathcal{I}_S&, \nonumber \\
\eta \in \mathbb{R}, \; x \in X, \; y_i \in \mathcal{S}^m_+ \; \forall i \in \mathcal{I}_S& \bigg\}. \nonumber
\end{align}
\end{proposition}

\begin{proof}
This follows directly from the variational representation of $\mathbb{CV}@R$ in \eqref{DefCVAR}. 
The expected-excess can be pushed into the restrictions by the same trick as in Proposition \ref{ExtensiveMeanRiskEE}. \qed
\end{proof}

As in in the risk-neutral case, problems \eqref{ExpEscessSSDPReformulation} and \eqref{CVaRReformulation} exhibit a block structure, i.e. there is no coupling constraint involving variables associated with different scenarios. This allows for a direct adaptation of the decomposition algorithms established for the expectation based model.

\begin{proposition}
Consider the problem
$$
\min \left\{Q_{\mathbb{E} + \rho \, \mathbb{V}@R_{\alpha}}(x) \; | \; x \in X \right\}
$$
with compact set $X$. This problem can be equivalently restated as the following SDP with binary variables
\begin{align}
\label{VaRReformulation}
\min_{\substack{x, v_1, \ldots, v_S, \\ y_1, \ldots, y_S, \\ \delta_1 , \ldots, \delta_S, \eta}} 
\bigg\{ (1+\rho) \, c \bullet x + \sum_{i = 1}^{S} \pi_i q \bullet y_i +
\rho \, \eta \; | \\
T \bullet x + W \bullet y_i = z_i \; \forall i \in \mathcal{I}_S&, \nonumber \\ 
\sum_{i = 1}^{S} \delta_i \, \pi_i \geq \alpha, \nonumber \\
\eta - q \bullet y_i \geq  (1- \delta_i)  \, M \; \forall i \in \mathcal{I}_S \nonumber \\
\eta \in \mathbb{R}, \; x \in X, \; \delta_i \in \{0,1\}, \;  y_i \in \mathcal{S}^m_+ \; \forall i \in \mathcal{I}_S& \bigg\} \nonumber
\end{align}
if $M \in \mathbb{R}$ is chosen sufficiently big.
\end{proposition}

\begin{proof}
As in the preceding propositions introduce a dummy variable $\eta$ to push $\mathbb{V}@R[\varphi(z - T \bullet x)]$
into the restrictions as $\eta \geq \mathbb{V}@R[\varphi(z - T \bullet x)]$ and minimize over $\eta$.
Note that $\eta \geq \mathbb{V}@R[ \varphi(z - T \bullet x) ]$ is equivalent to
\begin{align}\label{VaR1}
\mu( \varphi(z-T\bullet x) \leq \eta ) \geq \alpha.
\end{align}
As for given $x \in X$ feasible points to the second stage problem corresponding to realization $z_i$ are denoted as $y_i$,
\eqref{VaR1} can be rewritten as
\begin{align*}
\sum_{ i \in \mathcal{I}_S \; : \; q \bullet y_i \leq \eta } \pi_i \geq \alpha.
\end{align*}
This conditional summation can in turn be cast into inequalities with\\ binary variables $\delta_i$, $i \in \mathcal{I}_S$,
\begin{align*}
&\eta - q \bullet y_i \geq (1-\delta_i) \, M, \; i \in \mathcal{I}_S \\
&\sum_{i \in \mathcal{I}_S} \delta_i \, \pi_i \geq \alpha
\end{align*} 
if $M$ is chosen such that $\eta - q \bullet y_i < M$ for all feasible $y_i$ and all $\eta$ close to 
$\mathbb{V}@R[ \varphi(z_i - T \bullet x)]$. Since $- q \bullet y_i \leq - \varphi(z_i - T\bullet x)$ the existence
of $M$ follows from compactness of $X$, as $\max_{x \in X} \varphi(z_i - T \bullet x) < \infty$ for all $i \in \mathcal{I}_S$.  \qed
\end{proof}

Unlike the previous models, \eqref{VaRReformulation} does not decompose scenariowise due to the coupling constraint $\sum_{i = 1}^{S} \delta_i \, \pi_i \geq \alpha$, which involves variables from all scenarios. Furthermore, it has an additional binary variable for each scenario. Problems of a similar structure have been considered in the context of minimizing a weighted sum of the expectation and the probability of exceeding a fixed threshold in \cite{SchultzWollenberg2017}, where Lagrangian relaxation of the coupling constraint enables an approach based on Bender's decomposition. This direction seems also very promising for the algorithmic treatment of \eqref{VaRReformulation}.

\begin{proposition}
\label{PropSemidev}
\begin{equation*}
\min \left\{Q_{\mathbb{E} + \rho \,\text{$\mathbb{M}$ad}^+_p}(x) \; | \; x \in X \right\},
\end{equation*}
can be equivalently restated as
\begin{align*}
\min_{\substack{x, v_1, \ldots, v_S, \\ y_1, \ldots, y_S}} \bigg\{ c \bullet x + \sum_{i = 1}^{S} \pi_i q \bullet y_i + \rho \, \big( \sum_{i = 1}^{S} \pi_i v_i^p \big)^{\frac{1}{p}}  \; | \; T \bullet x + W \bullet y_i = z_i \; \forall i \in \mathcal{I}_S&, \nonumber \\ 
v \geq 0, \; v_i \geq c \bullet x + q \bullet y_i  - \sum_{ j = 1}^{S} \pi_j \, q \bullet y_j  \; \forall i \in \mathcal{I}_S&, \nonumber \\
x \in X, \; y_i \in \mathcal{S}^m_+ \; \forall i \in \mathcal{I}_S& \bigg\}. \nonumber
\end{align*}
\end{proposition}

\begin{proof}
Analogous to Proposition \ref{ExtensiveMeanRiskEE}. \qed
\end{proof}

Unlike \eqref{VaRReformulation}, the equivalent SDP in Proposition \ref{PropSemidev} contains an individual coupling constraint for each scenario. While Lagrangian relaxation still is possible, it remains to be examined whether this approach is sensible form a computational point of view.

\end{document}